\documentclass[a4paper,11pt]{amsart}
\usepackage{datetime}
\usepackage{a4wide}
\usepackage{eucal}
\usepackage{enumerate}
\usepackage{color}
\usepackage[colorlinks=true,linkcolor=blue]{hyperref}
\usepackage{bbm}
\usepackage[utf8]{inputenc}

\numberwithin{equation}{section}

\advance\textheight by 2cm

\theoremstyle{plain}
\newtheorem{thm}{Theorem}[section]
\newtheorem{cor}[thm]{Corollary}
\newtheorem{lem}[thm]{Lemma}
\newtheorem{prop}[thm]{Proposition}
\theoremstyle{definition}

\newtheorem{asu}[thm]{Assumptions}
\theoremstyle{remark}
\newtheorem{rem}[thm]{Remark}

\newtheorem{nota}[thm]{Notation}

\newcounter{stp}

\newcommand{\RR}{\mathbb{R}}
\newcommand{\CC}{\mathbb{C}}
\newcommand{\NN}{\mathbb{N}}
\newcommand{\eps}{\varepsilon}
\renewcommand{\phi}{\varphi}

\newcommand{\sL}{\mathcal{L}}
\newcommand{\sLX}{{\sL(X)}}
\newcommand{\sA}{\mathcal{A}}
\newcommand{\sZ}{\mathcal{Z}}
\newcommand{\sP}{\mathcal{P}}
\newcommand{\sU}{\mathcal{U}}
\newcommand{\sX}{\mathcal{X}}
\newcommand{\sG}{\mathcal{G}}
\newcommand{\Id}{Id}
\newcommand{\rL}{\mathrm{L}}
\newcommand{\rC}{\mathrm{C}}
\newcommand{\rW}{\mathrm{W}}
\newcommand{\p}{{\raisebox{1.3pt}{{$\scriptscriptstyle\bullet$}}}}

\newcommand{\U}{U}
\newcommand{\LA}{L_A}
\newcommand{\KD}{K_D}
\newcommand{\Y}{U}
\newcommand{\LpneU}{{\rL^p\bigl([0,1],\U\bigr)}}
\newcommand{\LpntU}{{\rL^p\bigl([0,t],\U\bigr)}}
\newcommand{\LrntU}{{\rL^r\bigl([0,t],\U\bigr)}}
\newcommand{\Cne}{{\rC[0,1]}}
\newcommand{\Cene}{{\rC^1[0,1]}}
\newcommand{\WnzpntU}{\rW^{2,p}_0\bigl([0,t],\U\bigr)}
\newcommand{\dX}{{\partial X}}
\newcommand{\dsX}{{\partial\sX}}
\newcommand{\dY}{{\partial Y}}

\newcommand{\Lene}{\rL^1[0,1]}
\newcommand{\Wzene}{\rW^{2,1}[0,1]}

\newcommand{\RlA}{R(\lambda,A)}
\newcommand{\RlAme}{R(\lambda,\Ame)}
\newcommand{\Tt}{(T(t))_{t\ge0}}
\newcommand{\Am}{{A_m}}

\newcommand{\St}{(S(t))_{t\ge0}}
\newcommand{\TP}{T_{BC}}
\newcommand{\TPt}{(\TP(t))_{t\ge0}}
\newcommand{\imp}{\Rightarrow}
\newcommand{\ds}{\,ds}
\newcommand{\dmr}{\,d\mu(r)}
\newcommand{\dnmr}{\,d|\mu|(r)}
\newcommand{\dt}{\,dt}
\newcommand{\dr}{\,dr}
\newcommand{\gb}{\operatorname{\omega_0}}
\newcommand{\rg}{\operatorname{rg}}
\newcommand{\inc}{\hookrightarrow}
\newcommand{\incc}{\overset{\text{c}}{\hookrightarrow}}
\newcommand{\XmeA}{X_{-1}^A}
\newcommand{\Fav}{\operatorname{F}}
\newcommand{\sXmesA}{\sX_{-1}^\sA}
\newcommand{\XeA}{X_{1}^A}
\newcommand{\Ame}{A_{-1}}
\newcommand{\Tme}{T_{-1}}
\newcommand{\tn}{{t_0}}
\newcommand{\Ft}{\mathcal{F}_t}
\newcommand{\ABC}{A_{BC}}
\renewcommand{\r}{\right}
\renewcommand{\l}{\left}

\newcommand{\dds}{\tfrac{d}{ds}}

\title[Perturbation of Analytic Semigroups and Applications]
{Perturbation of Analytic Semigroups and Applications to Partial Differential Equations}
\author{Martin Adler,  Miriam Bombieri and Klaus-Jochen Engel}
\address{Martin Adler\\
Arbeitsbereich Funktionalanalysis \\
Mathematisches Institut\\
Auf der Morgenstelle 10 \\
D-72076 T\"{u}bingen}
\email{maad@fa.uni-tuebingen.de}

\address{Miriam Bombieri\\
Arbeitsbereich Funktionalanalysis \\
Mathematisches Institut\\
Auf der Morgenstelle 10 \\
D-72076 T\"{u}bingen}
\email{mibo@fa.uni-tuebingen.de}

\address{Klaus-Jochen Engel\footnote{Corresponding author.}\\
Universit\`{a} degli Studi dell'Aquila \\
Dipartimento di Ingegneria e Scienze dell'Informazione e Matematica (DISIM)\\
Via Vetoio\\
I-67100 L'Aquila (AQ)}
\email{klaus.engel@univaq.it}

\keywords{Analytic semigroup, perturbation, generator, Favard space, fractional power}
\subjclass[2010]{47D06, 47A55,  34G10}
\dedicatory{Dedicated to Rainer Nagel on the occasion of his 75th birthday}

\date{\today, \currenttime}%

\begin{document}
\begin{abstract}
In a recent paper we presented a general perturbation result for generators of $C_0$-semigroups, see also Theorem~\ref{thm:main-gen} below. The aim of the present paper is to replace, in case the unperturbed semigroup is analytic, the various conditions appearing in this result by simpler assumptions on the domain and range of the operators involved. 
The power of our result to treat classes of PDE's systematically is illustrated by considering a generic example, a degenerate differential operator with generalized Wentzell boundary conditions and a reaction diffusion equation subject to  Neumann boundary conditions with distributed unbounded delay.
\end{abstract}
\maketitle

\section{Introduction}
Many partial differential equations can be abstractly written as a Cauchy problem of the form
\[
\tag{ACP}
\begin{cases}
{\tfrac{d}{dt}}\, x(t)=Gx(t),&t\ge0,\\
x(0)=x_0
\end{cases}
\]
where $G$ is an unbounded linear (differential) operator on a Banach space $X$, cf. \cite[Chap.~VI]{EN:00}. It is well-known that (ACP) is well-posed if and only if $G$ generates a $C_0$-semigroup on $X$, cf. \cite[Sect.~II.6]{EN:00}. Being concerned with the generator property of $G$ one basically relies on two methods, the Lumer-Phillips theorem in the dissipative- and the Hille--Yosida theorem in the general case. However, the latter is based on growth estimate of all powers of the resolvent of $A$ and can be verified explicitly only in very special cases.

\smallbreak
In order to check well-posedness of (ACP) for (non dissipative) operators $G$ where direct computations involving the resolvent are impossible to perform, one can try to split $G$ into a sum ``$G=A+P$\/'' for a simpler generator $A$ and a perturbation $P$ and then use some kind of perturbation theory to conclude that also $G$ generates a $C_0$-semigroup on $X$. 

\smallbreak
In  \cite{ABE:13} we presented an abstract result in this direction (see also Theorem~\ref{thm:main-gen} below), which can be interpreted as a purely operator-theoretic approach to former work by Weiss (cf. \cite[Thms.~6.1, 7.2]{Wei:94a}) and Staffans (cf. \cite[Thms.~7.1.2, 7.4.5]{Sta:05}) on the well-posedness of linear closed-loop systems.
More precisely, we considered operators $G$ of the form
\[
G=(A_{-1}+BC)|_X
\]
where $B\in\sL(U,X_{-1})$, $C\in\sL(Z,U)$, cf. Section~\ref{ssec:apr} for more details. 
Our former result was based on various ``admissibility'' conditions, cf. Notation~\ref{not:admiss} below, and  unified previous perturbation results due to Desch--Schappacher, Miyadera--Voigt and Greiner.

\smallbreak
In the present paper we replace these admissibility assumptions for generators $A$ of \emph{analytic} semigroups by simple inclusions concerning the range of $B$ and the domain of $C$ with respect to certain intermediate spaces.  We emphasize that our approach is highly versatile and allows treating various classes of differential equations in a unified and systematic way. Moreover, we point out that in contrast to other well-known results on the perturbation of analytic semigroups (see, e.g., \cite[Sect.~III.2]{EN:00}) we do \emph{not} assume that $P=BC$ is relatively bounded with respect to $A$.  On the contrary, in most cases the perturbation $P$ will change the domain of $A$, i.e., we may have $D(G)\ne D(A)$ as in the Examples in Section~\ref{sec:expls}.

\smallskip
This paper is organized as follows. In Section~\ref{sec:PoG} we first recall briefly the perturbation result from \cite{ABE:13} and then show how it can be simplified in the analytic case to obtain our main result, Theorem~\ref{lem:A^alpha-admiss-pair}. For its proof we use a characterization of analytic semigroups of angle $\theta\in(0,\frac{\pi}2]$, see Lemma~\ref{lem:char-analytic}, which might be of its own interest.
Then, in Section~\ref{sec:expls} we illustrate the power of our approach by three examples. First, we consider a generic example which significantly extends boundary perturbations considered by Greiner and Greiner--Kuhn, cf. Corollaries~\ref{cor:SW-gen} and \ref{cor:gen-expl-an}. Next, we apply our results to a degenerate second order differential operator on $\Cne$ with generalized Wentzell boundary conditions. Finally, we show the well-posedness of a reaction-diffusion equation on $\rL^p[0,\pi]$ subject to Neumann boundary conditions with distributed unbounded delay. In the Appendix we collect some results which are useful in order to verify the assumptions of our main results.

\smallskip
We mention that this is the first in a series of papers dedicated to applications of our perturbation result from \cite{ABE:13}. In forthcoming works we will treat, among other, perturbations of generators of cosine families, operator matrices, complete second order- and delay differential equations.

\section{Perturbation of Generators}
\label{sec:PoG}

\subsection{The Abstract Perturbation Result}\label{ssec:apr}

In this subsection we briefly recall the Weiss--Staffans type perturbation result from \cite{ABE:13}.

\smallbreak
On the Banach spaces $X$ and $\Y$ we consider the operators
\begin{itemize}
	\item $A:D(A)\subset X \to X$,
	\item $B\in\sL(\U,\XmeA)$, 
	\item $C\in\sL(Z,\Y)$
\end{itemize}
and assume that $A$ is the generator of a $C_0$-semigroup $\Tt$ on $X$. Here $X_1^A$ and $X_{-1}^A$ denote the inter- and extrapolated Sobolev spaces with respect to $A$ (cf. \cite[Sect.~II.5.a]{EN:00}). Moreover, $Z=D(C)$ is a Banach space such that
\[X_1^A\inc Z\inc X,\]
where ``$\inc$'' denotes a continuous injection.
For a triple $(A,B,C)$ as above and $t>0$ we indicate by
\[
\Ft^{(A,B,C)}:\LpntU\to\LpntU
\]
the associated \emph{input-output map}, i.e.,
\[
\bigl(\Ft^{(A,B,C)}u\bigr)(r)=
C\int_0^r T_{-1}(r-s) B u(s)\ds\quad\text{for }u\in\WnzpntU
\text{ and }r\in[0,t],
\]
where $(\Tme(t))_{t\ge0}$ denotes  the extrapolated semigroup generated by $\Ame$ on $\XmeA$.
Here for  $p\ge1$, $k\in\NN$, an interval $I\subseteq\RR$ such that $0\in I$ and a Banach space $V$ we define
\begin{equation*}
\rW^{k,p}_0(I,V):=
\bigl\{f\in\rW^{k,p}\bigl(I,V\bigr):f(0)=f'(0)=\ldots=f^{(k-1)}(0)=0\bigr\}.
\end{equation*}
Finally, we denote by $\rg(T)$ the range of a linear operator $T$.
Then by  \cite[Thm.~10]{ABE:13} the following holds.  
\begin{thm}\label{thm:main-gen}
Let $A$ generate a $C_0$-semigroup $\Tt$ on $X$, $B\in\sL(U,\XmeA)$ and $C\in\sL(Z,Y)$.
Moreover, assume that there exist $1\le p<+\infty$, $t>0$ and $M\ge0$ such that
\begin{alignat*}{2}
\text{(i)}\quad&\rg\bigl(\RlAme B\bigr)\subseteq Z&&\text{for some }\lambda\in\rho(A),\\
\text{(ii)}\quad&\int_0^{t} T_{-1}(t-s) B u(s)\ds \in X                    &&\text{for all }u\in\LpntU,\\
\text{(iii)}\quad&\int_0^{t}\bigl\|CT(s)x\bigr\|_U^p\ds \leq M\cdot\|x\|_X^p&&\text{for all }x\in D(A),\\
\text{(iv)}\quad&\int_0^t\Bigl\|C\int_0^r T_{-1}(r-s) B u(s)\ds\Bigr\|_U^p\dr\le M\cdot\|u\|_{p}^p\qquad&&\text{for all }u\in\WnzpntU,\\
\text{(v)}\quad&1\in\rho(\Ft^{(A,B,C)}).&&
\end{alignat*}
Then the operator
\begin{equation}\label{eq:def-A_BFC}
\ABC:=(A_{-1}+BC)|_{X},\quad D(\ABC)=\bigl\{x\in Z:(A_{-1}+BC)x\in X\bigr\}
\end{equation}
generates a $C_0$-semigroup $\TPt$ on the Banach space $X$. Moreover,  the perturbed semigroup verifies the \emph{variation of parameters formula}
\begin{equation*} 
\TP(t)x=T(t)x+\int_0^t\Tme(t-s)\cdot BC\cdot \TP(s)x\ds\quad\text{for all } t\ge0\text{ and }x\in D(\ABC).
\end{equation*}
\end{thm}

\begin{rem}\label{rem:B-admiss-estimate}
Using the closed graph theorem one can show that condition~(ii) in the previous result is equivalent to the estimate
\begin{equation}
\biggl\|\int_0^{t} T_{-1}(t-s) B u(s)\ds\biggr\|_X\le M\cdot\|u\|_{p}
\text{ for all }u\in\rW^{1,p}\bigl([0,t],U\bigr)
\end{equation}
for some $M\ge0$, cf. \cite[Rem.~2]{ABE:13}.
\end{rem}

\begin{nota}\label{not:admiss} 
It is convenient to use the following notions for operators $A$, $B$ and $C$ as above. Consider the conditions~(i)--(v) in Theorem~\ref{thm:main-gen}. Then
\begin{itemize}
\item the triple $(A,B,C)$ is called \emph{compatible} if (i) holds,
\item the operator $B$ is called \emph{$p$-admissible control operator} if (ii) holds,
\item the operator $C$ is called \emph{$p$-admissible observation operator} if (iii) holds,
\item the pair $(B,C)$ is called \emph{jointly $p$-admissible} if (i)--(iv)  hold,
\item the identity $Id_\U\in\sL(U)$ is called \emph{$p$-admissible feedback operator} if (v) holds.
\end{itemize}
\end{nota}

\subsection{Perturbation of Analytic Semigroups}\label{ssec:PAS}
While Theorem~\ref{thm:main-gen} is valid for arbitrary semigroups the following result is tailored for the analytic case. It substitutes the compatibility and admissibility conditions by inclusions of the range of $B$ and the domain of $C$ in certain intermediate spaces. 

\smallskip
In the sequel $\gb(A)$ denotes the growth bound of the semigroup generated by $A$, cf. \cite[Def.~I.5.6]{EN:00}, $\Fav^A_\alpha$ is the Favard space of $A$ of order $\alpha\in\RR$, see \cite[Sect.~II.5.b]{EN:00} and $(-A)^\gamma$ indicates the fractional power of order $\gamma\in\RR$ of $A$ as in \cite[Sect.~II.5.(c)]{EN:00}. Finally,  for a linear operator $T$ we write $[D(T)]:=(D(T),\|\p\|_T)$ with the graph norm $\|\p\|_T$ given by $\|x\|_T:=\|x\|+\|Tx\|$ for $x\in D(T)$.

\begin{thm}\label{lem:A^alpha-admiss-pair}
Let $(A,D(A))$ generate an analytic semigroup of angle $\theta\in(0,\frac\pi2]$ on $X$. Moreover, assume that for $B\in\sL(U,\XmeA)$ and $C\in\sL(Z,\Y)$ there exist $\lambda>\gb(A)$, $\beta\ge0$ and $\gamma>0$ such that
\begin{enumerate}[(i)]
\item $\rg(\RlAme B) \subseteq\Fav^A_{1-\beta}$,
\item $[D((\lambda-A)^\gamma)]\inc Z$,
\item $\beta+\gamma<1$.
\end{enumerate}
Then the following holds.
\goodbreak
\begin{enumerate}[(a)]
\item $(A,B,C)$ is compatible.
\item $B$ is a $p$-admissible control operator for all $p>\frac1{1-\beta}$, and, if $\beta=0$, then also for $p=1$.
\item $C$ is a $p$-admissible observation operator for all $p<\frac1\gamma$.
\item$(B,C)$ is jointly $p$-admissible for all $\frac1{1-\beta}<p<\frac1\gamma$, and, if $\beta=0$, then also for $p=1$.
\item for every $0<\eps<1-(\beta+\gamma)$ and $\frac1{1-\beta}\le p<\frac1\gamma$ there exists $M\ge0$ such that
\[
\bigl\|\Ft^{(A,B,C)}\bigr\|_p\le M\cdot t^{\eps}\quad\text{for all }0<t\le1.
\]
In particular, $Id_U$ is a $p$-admissible feedback operator for the triple $(A,B,C)$.
\end{enumerate}
Finally, the operator $(\Ame+BC)|_X$ generates an analytic $C_0$-semigroup of angle $\theta$ on $X$.
\end{thm}

In order to prove part (b) and (d) of the above result we need the following quite crippled version of Young's inequality which, however, perfectly fits our needs. Here, for two functions $K$ and $v$ on $(0,t_0]$ for some $t_0>0$ we define their convolution by
\[
(K*v)(t):=\int_0^tK(t-s)v(s)\ds,\quad t\in(0,t_0].
\]

\begin{lem}\label{lem:Young}
Let $K:(0,1]\to\sL(Y,X)$ be strongly continuous. Moreover, assume that $1\le p,q,r\le+\infty$ satisfy
$\frac1p+\frac1q=1+\frac1r$. 
If $k(\p):=||K(\p)||_{\sL(Y,X)}\in\rL^q[0,1]$ and $v\in\rC([0,1],Y)$, then $K*v\in\rL^r([0,1],X)$ and
\begin{equation}\label{eq:Young-ie}
\|K*v\|_r\le\|k\|_q\cdot\|v\|_p.
\end{equation}
\end{lem}

\begin{proof} We adapt the proof of \cite[Prop.~1.3.5]{ABHN:01} where convolutions on $\RR_+$ are considered and $K$ is assumed to be strongly continuous on $\RR_+$. %
Fix $0<t\le1$ and $v\in\rC([0,1],Y)$. Then $s\mapsto b(s):=K(t-s)v(s)$ is continuous on $(0,t)$, hence measurable. Note that $k(t-\p)\in\rL^{q}[0,t]\subseteq\rL^{1}[0,t]$, thus the estimate
\[
\|b(s)\|=\|K(t-s)v(s)\|\le k(t-s)\cdot\|v\|_\infty
\]
implies that $||b(\p)||$ is integrable on $[0,t]$. By Bochner's theorem (see \cite[Thm.~1.1.4]{ABHN:01}) this shows that $b$ is integrable and hence $(K*v)(t)$ exists for all $t\in[0,1]$. Next we show that $t\mapsto (K*v)(t)$ is continuous on $[0,1]$. Let $t\in[0,1]$ and $h\in\RR$ such that $t+h\in[0,1]$. Since $k\in\rL^q[0,1]\subseteq\rL^1[0,1]$ and $v\in\rC([0,1],Y)$ is uniformly continuous we conclude
\begin{align*}
\bigl\|(K*v)(t+h)&-(K*v)(t)\bigr\|\\
&\le\int_0^t k(s)\cdot\bigl\|v(t+h-s)-v(t-s)\bigr\|\ds
    +\biggl|\int_t^{t+h}k(s)\cdot\bigl\|v(s+h-s)\bigr\|\ds\biggr|\\
&\le\|k\|_1\cdot\sup_{s,s+h\in[0,1]}\bigl\|v(s+h)-v(s)\bigr\|
    +\biggl|\int_t^{t+h}k(s)\ds\biggr|\cdot\|v\|_{\infty}\\
&\to0\text{ as }h\to0.
\end{align*}
This shows that $K*v\in\rC([0,1],X)\subset\rL^r([0,1],X)$ and by the scalar-valued version of Young's inequality (see \cite[Sect.~IX.4, Expl.~1]{RS:75}) we finally obtain
\[
\|K*v\|_r\le \bigl\|k*\|v(\p)\|_Y\bigr\|_r\le\|k\|_q\cdot\|v\|_p
\]
as claimed.
\end{proof}

Now we are well-prepared to give the

\begin{proof}[Proof of Theorem~\ref{lem:A^alpha-admiss-pair}]
Note that for every $\lambda>\gb(A)$ we have $\Fav_{1-\beta}^A=\Fav_{1-\beta}^{A-\lambda}$. Hence, replacing if necessary $A$ by $A-\lambda$ we can assume without loss of generality that $\gb(A)<0$ and $\lambda=0$.

\smallbreak
(a) By \cite[Props.~II.5.14 \& 5.33]{EN:00} we have
\begin{equation}\label{eq:inc-a-f-x}
D\bigl((-A)^\alpha\bigr)\inc\Fav^A_\alpha\inc D\bigl((-A)^\delta\bigr)
\quad\text{ for all $1>\alpha>\delta>0$}.
\end{equation}
Since by assumption~(iii) we have $1-\beta>\gamma$, \eqref{eq:inc-a-f-x} and (ii) imply
\[
\rg\bigl(\Ame^{-1}B\bigr)\subseteq\Fav^A_{1-\beta}\subseteq D\bigl((-A)^\gamma\bigr)\subseteq Z=D(C),
\]
i.e., the triple $(A,B,C)$ is compatible.

\smallbreak
(b) Since $\Ame^{-1}B\in\sL(U,X)$ and $\Fav_{1-\beta}^A\inc X$, assumption~(i) and the closed graph theorem imply that $\Ame^{-1}B\in\sL(U,\Fav_{1-\beta}^A)$. Hence, for all $u\in\rC([0,1],U)\subset\LpneU$
\begin{equation}\label{eq:def-v}
v:=\Ame^{-1}B u\in\rC\bigl([0,1],\Fav_{1-\beta}^A\bigr)\subset\rL^p\bigl([0,1],\Fav_{1-\beta}^A\bigr).
\end{equation}
Since $\rg(T(t))\subseteq D(A^\infty)$ for all $t>0$, we can define
\[
K:(0,1]\to\sL\bigl(\Fav_{1-\beta}^A,X\bigr),\quad K(t):=AT(t).
\]
Then $K$ is strongly continuous on $(0,1]$ and by \cite[Prop.~II.5.13]{EN:00} there exists $M>0$ such that
\begin{equation*}
\bigl\|t^\beta K(t)x\bigr\|_X\le\sup_{s\in(0,1]}\bigl\|s^\beta AT(s)x\bigr\|_X\le M\cdot\|x\|_{\Fav_{1-\beta}^A}
\quad\text{ for all $x\in\Fav_{1-\beta}^A$}.
\end{equation*}
This implies that
\begin{equation}\label{eq:esti-AT(t)-Y_beta}
k(t):=\|K(t)\|_{\sL(\Fav_{1-\beta}^A,X)}\le M\cdot t^{-\beta}\quad\text{for all }t\in(0,1].
\end{equation}
Hence, $k\in\rL^q[0,1]$ if $\beta\cdot q<1$, i.e.,
\begin{equation*}
k\in\rL^q[0,1]\quad\text{if}\quad
\begin{cases}
q<\frac1\beta&\text{and }\beta>0,\text{ or}\\
q\ge 1&\text{and }\beta=0.
\end{cases}
\end{equation*}
Now we choose $r=+\infty$ in Young's inequality from Lemma~\ref{lem:Young}. Then $q=\frac p{p-1}$ and from \eqref{eq:Young-ie} it follows that there exists $M\ge0$ such that for all $u\in\rC([0,1],U)$
\begin{align*}
\biggl\|\int_0^1\Tme(1-s)Bu(s)\ds\biggr\|_X
&=\bigl\|(K*v)(1)\bigr\|_X
\le \|K*v\|_\infty\\
&\le M\cdot \|k\|_q\cdot\|u\|_p
\end{align*}
provided
\begin{equation*}
\left\{
\begin{aligned}
\frac p{p-1}=q<\frac1\beta&\text{ and }\beta>0&&\iff\quad p>\frac1{1-\beta}\text{ and }\beta>0,\text{ or}\\
\frac p{p-1}=q\ge 1&\text{ and }\beta=0&&\iff\quad p\ge1\text{ and }\beta=0.
\end{aligned}
\right.
\end{equation*}
Since $\rW^{1,p}([0,1],U)\subset\rC([0,1],U)$, the assertion follows from Remark~\ref{rem:B-admiss-estimate}.

\smallbreak
(c) For all $t>0$ we have by (ii)
\[
\bigl\| CT(t)\bigr\|_{\sL(X,\Y)}\le\bigl\|C(-A)^{-\gamma}\bigr\|_{\sL(X,\Y)}\cdot\bigl\|(-A)^\gamma T(t)\bigr\|_{\sLX}.
\]
Since by \cite[Lem.~11.36]{RR:93} there exists $M\ge0$ such that
\begin{equation}\label{eq:est-AT(t)-aHG}
\bigl\|(-A)^\gamma T(t)\bigr\|_{\sLX}\le M\cdot t^{-\gamma}\quad\text{for all }t\in(0,1],
\end{equation}
we conclude that $C$ is a $p$-admissible observation operator for all $p<\frac1\gamma$.

\smallbreak
(d) Since $\rg(T(t))\subseteq D(A^\infty)$ for all $t>0$ we can define
\[
L:(0,1]\to\sL\bigl(\Fav_{1-\beta}^A,X\bigr),\quad L(t):=(-A)^{1+\gamma}T(t).
\]
Then $L$ is strongly continuous on $(0,1]$.
Using \eqref{eq:esti-AT(t)-Y_beta} and \eqref{eq:est-AT(t)-aHG} we obtain for suitable $M\ge0$
\begin{align}
l(t):=\bigl\|L(t)\bigr\|_{\sL(\Fav_{1-\beta}^A,X)}\notag
&\le\bigl\|(-A)^\gamma T\bigl(\tfrac t2\bigr)\bigr\|_{\sLX}\cdot \bigl\|AT\bigl(\tfrac t2\bigr)\bigr\|_{\sL(\Fav_{1-\beta}^A,X)}\\\label{eq:est-L(t)}
&\le M\cdot t^{-(\beta+\gamma)}.
\end{align}
Now choose in Young's inequality from Lemma~\ref{lem:Young} $p=\frac1{1-\beta}\le r<\frac1\gamma$. Then we obtain $\frac1q=\beta+\frac1r>\beta+\gamma$ and hence
\[
q\cdot(\beta+\gamma)<1
\]
which by \eqref{eq:est-L(t)} implies that $l\in \rL^q[0,1]$. Thus, by \eqref{eq:Young-ie} there exists $M\ge0$ such that the input-output map $\Ft:=\Ft^{(A,B,C)}$ satisfies for all $u\in\rC([0,1],U)$
\begin{align*}
\bigl\|\Ft u\bigr\|_r=
\biggl(\int_0^1\biggl\|C\int_0^t\Tme&(t-s)Bu(s)\ds\biggr\|_\Y^{r}\!dt\biggr)^{\frac1{r}}\\
&\le\bigl\|C(-A)^{-\gamma}\bigr\|\cdot\biggl(\int_0^1\biggl\|\int_0^t(-A)^{1+\gamma}\Tme(t-s)\cdot\Ame^{-1} Bu(s)\ds\biggr\|_X^{r}\!dt\biggr)^{\frac1{r}}\\
&\le M\cdot\bigl\|(L*v)\bigr\|_{r}\\
&\le M\cdot \|l\|_q\cdot\|u\|_{\frac1{1-\beta}}
\end{align*}
where $v\in\rC\bigl([0,1],\Fav_{1-\beta}^A\bigr)$ is given by \eqref{eq:def-v}. This shows that for every $\frac1{1-\beta}\le r<\frac1\gamma$ the input-output map has a unique bounded extension
\begin{equation}\label{eq:F_t:Lp->Lr}
\Ft:\rL^{\frac1{1-\beta}}\bigl([0,t],U\bigr)\to\rL^{r}\bigl([0,t],\Y\bigr).
\end{equation}
Since $\rL^r\bigl([0,t],U\bigr)\inc \rL^{\frac1{1-\beta}}\bigl([0,t],U\bigr)$ for $r\ge{\frac1{1-\beta}}$, this together with (b) and (c) proves that the pair $(B,C)$ is jointly $p$-admissible for all $p\in(\frac1{1-\beta},\frac1\gamma)$ and in case $\beta=0$ also for $p=1$.

\smallbreak
(e) By Jensen's inequality we have for all $1\le p\le r<+\infty$ and $u\in\LrntU\subseteq\LpntU$
\begin{equation*} 
\|u\|_p\le t^{\frac1p-\frac1r}\cdot\|u\|_r.
\end{equation*}
This combined with \eqref{eq:F_t:Lp->Lr} gives for all $\frac1{1-\beta}\le p\le r<\frac1\gamma$ and $u\in\LrntU$ that
\begin{align*}
t^{-\frac1p+\frac1r}\cdot\|\Ft u\|_p\le\|\Ft u\|_r\le M\cdot\|u\|_{\frac1{1-\beta}}\le M\cdot t^{{1-\beta}-\frac1p}\cdot\|u\|_p.
\end{align*}
For given $0<\eps<1-(\beta+\gamma)$ we take $r:=\frac1{1-\beta-\eps}\in\bigl(\frac1{1-\beta},\frac1\gamma\bigr)$ and obtain by density of $\LrntU$ in $\LpntU$ that
\[
\|\Ft\|_p\le M\cdot t^{{1-\beta}-\frac1r}
\le M\cdot t^{\eps}
\]
as claimed. Clearly (d) and (e) combined with Theorem~\ref{thm:main-gen} imply that $(\Ame+BC)|_X$ generates a $C_0$-semigroup. This semigroup is analytic of angle $\theta$ by the following two lemmas. More precisely, Lemma~\ref{lem:F-D-analytic} allows us to repeat the above reasoning for $A$, $B$, $C$ replaced by $e^{i\phi}A$, $e^{i\phi}B$, $C$ to obtain that also $e^{i\phi}A_{BC}$ is a generator of a $C_0$-semigroup on $X$ for all $\phi\in(-\theta,\theta)$. By Lemma~\ref{lem:char-analytic} this implies the assertion.
\end{proof}

In the following for an operator $A$ and $\phi\in\RR$ we use the notation
\[
A_\phi:=e^{i\phi}A.
\]

\begin{lem}\label{lem:char-analytic}
Let $0<\theta\le\frac\pi2$.
Then $A$ generates an analytic semigroup of angle $\theta$ if and only if  $A_\phi$ generates a $C_0$-semigroup for every $\phi\in(-\theta,\theta)$. 
\end{lem}

\begin{proof} Assume first that  $A$ generates an analytic semigroup $(T(z))_{z\in\Sigma_\theta\cup\{0\}}$ of angle $\theta$ where $\Sigma_\theta$ denotes the open sector
\[
\Sigma_\theta:=\bigl\{z\in\CC\setminus\{0\}:|\arg(z)|<\theta\bigr\}.
\]
Then it is clear that for every $\phi\in(-\theta,\theta)$ the operators $T_\phi(t):=T(e^{i\phi}t)$ define a strongly continuous semigroup $(T_\phi(t))_{t\ge0}$ with generator $A_\phi$, cf. \cite[Prop.~3.7.2.(c)]{ABHN:01}.

\smallbreak
Conversely, assume that $A_\phi$ generates a $C_0$-semigroup $(T_\phi(t))_{t\ge0}$ for every $\phi\in(-\theta,\theta)$. Then by  \cite[Thm.~II.4.6.(b)]{EN:00} the operator $A$ generates an analytic semigroup $(T(z))_{z\in\Sigma_{\theta'}\cup\{0\}}$ of some angle $\theta'>0$. If $\theta'\ge\theta$ we are done, hence assume that $\theta'<\theta$. Then we have to show that the map $z\mapsto T(z)$ can be extended analytically from $\Sigma_{\theta'}$ to the open sector $\Sigma_\theta$.

\smallbreak
To this end we fix some $\phi\in(\theta',\theta)$ and consider the two projections on the complex plane %
$P_{\pm\phi}:\CC\to\CC$ onto $e^{\pm i\phi}\cdot\RR$ along $e^{\mp\phi i}\cdot\RR$. Then for $z\in\bar\Sigma_\phi$ we put $r_\pm(z):=e^{\mp i\phi}\cdot P_{\pm\phi}z\ge0$. Since $P_{\pm\phi}=1-P_{\mp\phi}$, this implies
$z=r_+(z)\cdot e^{i\phi}+r_-(z)\cdot e^{-i\phi}$. Using this representation of $z$ we define
\begin{equation}\label{eq:def-T(z)}
\tilde T:\bar\Sigma_\phi\to\sLX,\quad\tilde T(z):=T_{\phi}\bigl(r_+(z)\bigr)\cdot T_{-\phi}\bigl(r_-(z)\bigr).
\end{equation}
Since the resolvents of $A_{\pm\phi}$ commute also the semigroups $(T_{\pm\phi}(t))_{t\ge0}$ commute. Using this fact and the equations $r_{\pm}(z+w)=r_{\pm}(z)+r_{\pm}(w)$ it follows that
\[
\tilde T(z)\cdot\tilde T(w)=\tilde T(z+w)\quad\text{for all }z,w\in\bar\Sigma_\phi.
\]
Next we show that $(\tilde T(z))_{z\in\bar\Sigma_\phi}$  is strongly continuous on the closed sector $\bar\Sigma_\phi$. To this end choose $M,\omega>0$ such that $\|T_{\pm\phi}(t)\|\le M\cdot e^{\omega t}$ for all $t\ge0$. Then from the continuity of $r_\pm(\p)$ and the fact that $r_\pm(z)\le\|P_{\pm\phi}\|\cdot|z|$ we obtain for $x\in X$ and $z,w\in\bar\Sigma_\phi$
\begin{align*}
\bigl\|\tilde T(z)x-\tilde T(w)x\bigr\|
&\le\bigl\|T_{\phi}\bigl(r_+(z)\bigr)\cdot\bigl[T_{-\phi}\bigl(r_-(z)\bigr)-T_{-\phi}\bigl(r_-(w)\bigr)\bigr]x\bigr\|\\
&\qquad\qquad\qquad+
\bigl\|\bigl[T_{\phi}\bigl(r_+(z)\bigr)-T_{\phi}\bigl(r_+(w)\bigr)\bigr]\cdot T_{-\phi}\bigl(r_-(w)\bigr)x\bigr\|\\
&\le Me^{\omega\bigl(\|P_\phi\|+\|P_{-\phi}\|\bigr)\cdot|z|}\cdot\Bigl(\bigl\|T_{-\phi}\bigl(r_-(z)\bigr)x-T_{-\phi}\bigl(r_-(w)\bigr)x\bigr\|\\
&\hskip40mm+
\bigl\|T_{\phi}\bigl(r_+(z)\bigr)x-T_{\phi}\bigl(r_+(w)\bigr)x\bigr\|\Bigr)\\
&\to0\text{ as $w\to z$ in $\bar\Sigma_\phi$.}
\end{align*}
Hence, $(\tilde T(z))_{z\in\bar\Sigma_\phi}$  is strongly continuous as claimed. This implies in particular that for every $\psi\in[-\phi,\phi] $ the restriction
\[
\tilde T_\psi(t):=\tilde T\bigl(e^{i\psi}t\bigr),\quad t\ge0
\]
defines a $C_0$-semigroup on $X$. Next we compute its generator $\tilde A_\psi$. Let
\[
r_\pm:=r_\pm\bigl(e^{i\psi}\bigr),\quad\text{i.e.}\quad
e^{i\psi}=r_+\cdot e^{i\phi}+r_-\cdot e^{-i\phi}.
\]
Then by definition of $\tilde T(z)$ in \eqref{eq:def-T(z)} we have $\tilde T_\psi(t)=T_\phi(r_+ t)\cdot T_{-\phi}(r_- t)$. Hence, for $x\in D(A)$ we obtain
\begin{align*}
\frac{d}{dt}\,\tilde T_\psi(t)x
&=r_+A_\phi\cdot T_\phi(r_+ t)\cdot T_{-\phi}(r_- t)x+r_-A_{-\phi}\cdot T_\phi(r_+ t)\cdot T_{-\phi}(r_- t)x\\
&=\bigl(r_+e^{i\phi}A+r_-e^{-i\phi}A\bigr)\cdot\tilde T(z)x=e^{i\psi}A\cdot\tilde T(z)x.
\end{align*}
This implies $e^{i\psi}A\subseteq\tilde A_\psi$ and since $\rho(e^{i\psi}A)\cap\rho(\tilde A_\psi)\ne\emptyset$ we obtain $\tilde A_\psi=e^{i\psi}A=A_\psi$.
Since a generator uniquely determines the generated semigroup we conclude that
\[
T(z)=\tilde T(z)\quad\text{for all }z\in\Sigma_{\theta'},
\]
i.e., $(\tilde T(z))_{z\in\bar\Sigma_\phi}$ is a strongly continuous extension of  $(T(z))_{z\in\Sigma_{\theta'}\cup\{0\}}$. For this reason from now on we can drop the tilde and write $T(z)=\tilde T(z)$ for all $z\in\bar\Sigma_\phi$. 

\smallskip
Summing up, we showed that $A$ generates a semigroup $(T(z))_{z\in\bar\Sigma_\phi}$ which is strongly continuous on $\bar\Sigma_\phi$ and analytic on $\Sigma_{\theta'}$.
It remains to show that $(T(z))_{z\in\bar\Sigma_\phi}$ is analytic on $\Sigma_\phi$. To this end note that for each $r>0$ 
\begin{align*}
z&\mapsto T\bigl(re^{\pm i\phi}\bigr)\cdot T(z)=T\bigl(re^{\pm i\phi}+z\bigr)\quad\text{is analytic on }\Sigma_{\theta'}\quad\Longrightarrow\\
z&\mapsto T(z)\quad\text{is analytic on }re^{\pm i\phi}+\Sigma_{\theta'}.
\end{align*}
Since
\[
\Sigma_\phi=\bigcup_{r>0}\,\Bigl(re^{\pm i\phi}+\Sigma_{\theta'}\Bigr)
\]
this implies that $(T(z))_{z\in\bar\Sigma_\phi}$ is analytic on the whole open sector $\Sigma_\phi$ as claimed. Recall that $\phi\in(\theta',\theta)$ was arbitrary. Thus, from
\[
\Sigma_\theta=\bigcup_{\phi\in(-\theta,\theta)}\Sigma_\phi
\]
we finally conclude that $\Tt$ can be extended to an analytic semigroup $(T(z))_{z\in\bar\Sigma_\theta\cup\{0\}}$, i.e., is analytic of angle (at least) $\theta$.
\end{proof}

\begin{lem}\label{lem:F-D-analytic}
Let $A$ generate an analytic semigroup of angle $\theta\in(0,\frac\pi2)$. Moreover, let
$\phi\in(-\theta,\theta)$ and $\lambda>0$ such that $\gb(A-\lambda),\gb(A_\phi-\lambda)<0$. 
Then for all $\alpha\in(0,1]$ one has
\begin{equation}\label{eq:equal-D-Fav_A_phi}
D\bigl((\lambda-A)^\alpha\bigr)=D\bigl((\lambda-A_\phi)^\alpha\bigr)
\qquad\text{and}\qquad
\Fav^A_\alpha=\Fav^{A_\phi}_\alpha.
\end{equation}
\end{lem}

\begin{proof}
Note that by the previous result $A_\phi$ generates an analytic semigroup. Without loss of generality we assume that $\lambda=0$.

\smallbreak
To show the first equality in \eqref{eq:equal-D-Fav_A_phi} fix some $\alpha\in(0,1)$. Then by the definition of $(-A)^{-\alpha}$ (see, e.g. \cite[Def.~II.5.25]{EN:00}), the equality 
\begin{equation}\label{eq:Res-Aphi}
R(\lambda,A_\phi)=e^{-i\phi}R\bigl(e^{-i\phi}\lambda,A\bigr)
\end{equation}
and Cauchy's integral theorem it follows that
$
(-A_\phi)^{-\alpha}=e^{-i\phi\alpha}\cdot(-A)^{-\alpha}.
$
This implies 
\[D((-A)^\alpha)=\rg((-A)^{-\alpha})=\rg((-A_\phi)^{-\alpha})=D((-A_\phi)^\alpha)\]
for $\alpha\in(0,1)$ while for $\alpha=1$ it is obviously satisfied.

\smallskip
Next we show that $\Fav^A_\alpha\subseteq\Fav^{A_\phi}_\alpha$ for $\alpha\in(0,1]$. Let $x\in\Fav^A_\alpha$. Then by \eqref{eq:Res-Aphi}, the resolvent equation, the Hille--Yosida theorem for $A_\phi$ and \cite[Prop.~II.5.12]{EN:00} we conclude that
\begin{align*}
\sup_{\lambda>0}\bigl\|\lambda^\alpha A_\phi R(\lambda,A_\phi)x\bigr\|
&=\sup_{\lambda>0}\bigl\|\lambda^\alpha A\bigl(R(e^{-i\phi}\lambda,A)-\RlA\bigr)x+\lambda^\alpha A R(\lambda,A)x\bigr\|\\
&\le\sup_{\lambda>0}\bigl\|(1-e^{-i\phi})\lambda R(e^{-i\phi}\lambda,A)\lambda^\alpha A\RlA x\bigr\|+\sup_{\lambda>0}\bigl\|\lambda^\alpha A R(\lambda,A)x\bigr\|\\
&\le\l(1+\sup_{\lambda>0}\bigl\|(1-e^{-i\phi})\lambda R(\lambda,A_\phi)\bigr\|\r)\cdot\sup_{\lambda>0}\bigl\|\lambda^\alpha A\RlA x\bigr\|
<+\infty.
\end{align*}
Again by \cite[Prop.~II.5.12]{EN:00} this implies that $x\in\Fav^{A_\phi}_\alpha$, hence $\Fav^A_\alpha\subseteq\Fav^{A_\phi}_\alpha$. In order to show the converse inclusion note that $A=e^{-i\phi}A_\phi$. The assertion then follows as above by interchanging the roles of $A$ and $A_\phi$ and  substituting $\phi$ by $-\phi$.
\end{proof}

We conclude this section by the following observation which sometimes helps to check the admissibility of $B$ or $C$.

\begin{rem}
In the situation of Theorem~\ref{lem:A^alpha-admiss-pair} the implications (i)$\imp$(b) and (ii)$\imp$(c) hold.
\end{rem}

\section{Examples}
\label{sec:expls}

\subsection{The Generic Example}
\label{subsec:GE}
Many concrete examples fit into the following general framework which generalizes boundary perturbations in the sense of Greiner, cf. \cite{Gre:87}.

\smallbreak
We start with a Banach space $X$ and a  linear ``maximal operator''\footnote{``Maximal'' in the sense of a ``big'' domain, e.g., a differential operator without boundary conditions.} $A_m:D(A_m)\subseteq X\to X$.
In order to single out a restriction $A$ of $A_m$ we take a Banach spaces $\dX$, called ``space of boundary conditions'', and a linear ``boundary operator''
$L:D(A_m)\to\dX$
and define
\begin{equation}\label{eq:def-A-kerK}
A\subseteq A_m,\quad D(A)=\bigl\{x\in D(A_m): Lx=0\bigr\}=\ker(L).
\end{equation}
Next we perturb $A$ in the following way.
For a Banach space $Z$ satisfying $D(A_m)\subseteq Z$ and $\XeA\inc Z\inc X$, and operators $P\in\sL(Z,X)$ and $\Phi\in\sL(Z,\dX)$ we consider
\begin{equation}\label{eq:G-one-dim}
G:=A_m+P,\quad D(G):=\bigl\{x\in D(A_m): Lx=\Phi x\bigr\}.
\end{equation}
Hence, $G$ can be considered as a twofold perturbation of $A$,
\begin{itemize}
\item by the operator $P$ to change its action, and
\item by the operator $\Phi$ to change its domain.
\end{itemize}
We note that in \cite{Gre:87} the operator $\Phi:X\to\dX$ has to be bounded and $P=0$.
Below we will show that this example fits into our framework for unbounded operators $\Phi$ and $P$, too. To this end we first make the following

\begin{asu}\label{assump:Dir-op-ex}
\begin{enumerate}[(a)]
\item $A$ generates a $C_0$-semigroup $\Tt$ on $X$.
\item For some $\mu\in\CC$ the restriction
\[
L|_{\ker(\mu-A_m)}:\ker(\mu-A_m)\to\dX
\]
is invertible with bounded inverse
\begin{equation*}
L_\mu:=\bigl(L|_{\ker(\mu-A_m)}\bigr)^{-1}\in
\sL(\dX,X).
\end{equation*}
\end{enumerate}
\end{asu}

Next we elaborate on the so-called abstract \emph{Dirichlet operator} $L_\mu$ which plays a crucial role in this approach.

\begin{prop}\label{prop:L_lambda-rhoA}
Let Assumption~\ref{assump:Dir-op-ex}.(b) be satisfied. Then for all $\lambda\in\rho(A)$
\[
L|_{\ker(\lambda-A_m)}:\ker(\lambda-A_m)\to\dX
\]
is invertible with bounded inverse given by
\begin{equation}\label{eq:def-L_mu}
L_\lambda=(\mu-A)\RlA L_\mu\in
\sL(\dX,X).
\end{equation}

\end{prop}

\begin{proof} Let $\tilde L_\lambda\in\sL(\dX,X)$ be the operator defined by the right-hand-side of \eqref{eq:def-L_mu}. Then the identity $\tilde L_\lambda=\bigl(\Id+(\mu-\lambda)\RlA\bigr) L_\mu$ implies that $\rg(\tilde L_\lambda)\subseteq\ker(\mu-A_m)+D(A)\subseteq D(A_m)$ and $L\tilde L_\lambda=Id_\dX$. Moreover, for $x\in\dX$
\begin{align*}
(\lambda-A_m)\tilde L_\lambda x
&=(\lambda-A_m)L_\mu x+(\mu-\lambda)(\lambda-A_m)\RlA L_\mu x\\
&=(\lambda-\mu)L_\mu x+(\mu-\lambda) L_\mu x=0,
\end{align*}
i.e., $\rg(\tilde L_\lambda)\subseteq\ker(\lambda-A_m)$.
Summing up this proves that $L:\ker(\lambda-A_m)\to\dX$ is surjective with right-inverse $\tilde L_\lambda$. To show injectivity assume that $x\in\ker(\lambda-A_m)\cap\ker(L)$. Then $x\in D(A)$ and $(\lambda-A)x=0$ which implies $x=0$ since $\lambda\in\rho(A)$.
\end{proof}

Note that by the previous result $L_\lambda=\RlAme(\mu-\Ame) L_\mu$, hence
the operator
\begin{equation}\label{eq:def-LA}
\LA:=(\mu-\Ame)L_\mu=(\lambda-\Ame)L_\lambda\in\sL(\dX,\XmeA)
\end{equation}
is independent of $\lambda\in\rho(A)$.

\smallskip
The following result gives sufficient conditions implying Assumption~\ref{assump:Dir-op-ex}.(b). For a proof we refer to \cite[Lem.~1.2]{Gre:87} and \cite[Lem.~2.2]{CENN:03}.

\begin{lem}\label{lem:Dir-op}
If $A_m$ or $\binom{A_m}{L}$ is closed and $L$  is surjective, then for every $\lambda\in\rho(A)$  Assumption~\ref{assump:Dir-op-ex}.(b) is satisfied. 
\end{lem}
Using the operator $\LA$  from \eqref{eq:def-LA} we obtain the following representation of $G$ from \eqref{eq:G-one-dim}.

\begin{lem}\label{lem:G-as-SWp}
We have
\begin{equation}\label{eq:rep-M-K-Theta}
G=\bigl(\Ame+P+\LA\cdot\Phi\bigr)\big|_{X}.
\end{equation}
\end{lem}

\begin{proof}
Denote by $\tilde G$ the operator defined by the right-hand-side of \eqref{eq:rep-M-K-Theta} and fix some $\lambda\in\rho(A)$. Then for $x\in Z$ we have
\begin{align}\notag
x\in D(\tilde G)&\iff (\Ame-\lambda)\bigl(\Id-L_\lambda\Phi\bigr)x+(P+\lambda)x\in X\\\notag
&\iff \bigl(\Id-L_\lambda\Phi\bigr)x\in D(A)=\ker L\\
&\iff Lx=\Phi x\label{eq:iff-D(G)}\\
&\iff x\in D(G),\notag
\end{align}
where in \eqref{eq:iff-D(G)} we used that $ x=\bigl(\Id-L_\lambda\Phi\bigr)x+L_\lambda\Phi x\in D(A)+\ker(\lambda-A_m)\subseteq D(A_m)$.
Moreover, for $x\in D(G)$ we obtain
\begin{align*}
\tilde Gx&=(\Am-\lambda)\bigl(\Id-L_\lambda\Phi\bigr)x+(P+\lambda)x\\
&=(\Am-\lambda)x +(P+\lambda)x\\
&=(\Am+P)x=Gx,
\end{align*}
hence $G=\tilde G$ as claimed.
\end{proof}

In order to represent $G$ given in  \eqref{eq:rep-M-K-Theta} as $\ABC$ like in \eqref{eq:def-A_BFC}, we define the product space
\begin{equation}\label{eq:def-U}
U:=X\times\dX
\end{equation}
and the operators
\begin{equation}\label{eq:def-B-C}
B:=\bigl(Id_X,\LA\bigr)\in\sL(U,\XmeA)
\quad\text{and}\quad
C:=\tbinom{P}{\Phi}\in\sL(Z,\Y).
\end{equation}
Then a simple computation shows the following.

\begin{lem} 
The triple $(A,B,C)$ given by \eqref{eq:def-A-kerK}, \eqref{eq:def-B-C} is compatible.
Moreover, $G$ in \eqref{eq:G-one-dim} can be written as $G=\ABC$.
\end{lem}

By applying Theorem~\ref{thm:main-gen} to this situation we obtain the following result.

\begin{cor}\label{cor:SW-gen}
Assume that for some $1\le p < +\infty$ the pairs
$(\LA,P)$ and $(\LA,\Phi)$ are jointly $p$-admissible for $A$ and that there exists  $t>0$ such that $1\in\rho(\Ft^{(A,B,C)})$ where
\begin{equation*} 
\Ft^{(A,B,C)}=
\begin{pmatrix}
\Ft^{(A,Id_X,P)}&\Ft^{(A,\LA,P)}\\
\Ft^{(A,Id_X,\Phi)}&\Ft^{(A,\LA,\Phi)}
\end{pmatrix}
\in\sL\bigl(X\times\dX\bigr).
\end{equation*}
Then $G$ given by \eqref{eq:G-one-dim} generates a $C_0$-semigroup on $X$. Here the condition $1\in\rho(\Ft^{(A,B,C)})$ is in particular satisfied if $p>1$ and $1\in\rho(\Ft^{(A,\LA,\Phi)})$.
\end{cor}

\begin{proof}
We only have to verify the assertion concerning the invertibility of $Id-\Ft^{(A,B,C)}$ for $p>1$. This, however, follows immediately from Lemma~\ref{lem:norm-sFt}.(b).
\end{proof}

We note that Corollary~\ref{cor:SW-gen} can be regarded as an operator theoretic extension of the main result  in \cite{HMR:15} (Theorem~4.1) which is formulated in the language of systems theory and where $P=0$ is assumed.

\smallbreak
If in the above situation $A$ generates an analytic semigroup, then from Theorem~\ref{lem:A^alpha-admiss-pair} and Lemma~\ref{lem:char-1-admiss} we obtain the following simplification.

\begin{cor}\label{cor:gen-expl-an}
Let $A$ generate an analytic semigroup of angle $\theta\in(0,\frac\pi2]$ on $X$. If there exist $\lambda>\gb(A)$, $\beta\ge0$ and $\gamma>0$ such that
\begin{enumerate}[(i)]
\item $\rg(L_\lambda)\subseteq\Fav^A_{1-\beta}$,
\item $[D((\lambda-A)^\gamma)]\inc Z$,
\item $\beta+\gamma<1$,
\end{enumerate}
then $G$ given by \eqref{eq:G-one-dim} generates an analytic semigroup of angle $\theta$  on $X$.
\end{cor}

\begin{rem}\label{rem:G-K}
The previous corollary improves  \cite[Thm.~2.6.(c)]{GK:91} where, by means of a resolvent estimate, a similar result in the context of abstract Hölder spaces is proved. In contrast to our approach, the authors don't obtain information on the angle of analyticity. Moreover, their approach is not applicable to problems
like Example~\ref{sec:rde} where only part of the system is governed by an analytic semigroup, but the semigroup associated to the whole system is not analytic, cf. Remark~\ref{rem:an/non-an}.
\end{rem}

\subsection{A Degenerate Second-Order Differential Operator on $\mathbf{C[0,1]}$ with Generalized Wentzell Boundary Conditions}

As a first concrete application of our approach we prove the following generation result.

\begin{thm}\label{thm:expl-Wentzell}
Let $a\in\Cne$ such that $\frac1a\in\Lene$ and $[0,1]\ni s\mapsto\int_0^s\frac1{a(r)}\dr$ is Hölder-continuous of exponent $\delta\in(0,1]$.
Moreover, define $A_m$ on $\Cne$ by
\begin{alignat}{2} 
&A_m f:=a\cdot f'',\quad & &D(A_m):=\bigl\{f\in\Cne\cap\rC^2(0,1):
a\cdot
f''\in\Cne\bigr\}.\\
\intertext{Then $D(A_m)\subset\Cene$ and for all $b,c\in\Cne$ and
$\Phi\in\sL(\Cene,\CC^2)$ the operator}
\label{eq:se.sdd.2} 
&Gf:=A_m f+bf'+cf,\quad& &D(G):=\bigl\{f\in
D(A_m):\tbinom{(A_m f)(0)}{(A_m f)(1)}=\Phi f\bigr\}
\end{alignat}
generates a compact, analytic semigroup of angle $\frac{\pi}{2}$
on $\Cne$.
\end{thm}

\begin{proof}
Since $\frac{1}{a}\in\Lene$, it follows  for $f\in D(\Am)$ that $f''=\frac{1}{a}\cdot\Am f\in\Lene$. Hence,
$f\in\Wzene\subset\Cene$ and therefore $D(\Am)\subset\Cene$.
Next we define the operator $A\subset\Am$ with domain
\[
D(A):=\Bigl\{f\in D(A_m):\tbinom{(A_m f)(0)}{(A_m f)(1)}=0\Bigr\}.
\]
Moreover, we choose $X:=\Cne$, $\dX:=\CC^2$, $Z:=\Cene$,
$L:=\binom{\delta_0\Am}{\delta_1\Am}:D(\Am)\to\dX$ and $P:=b(\p)\dds+c(\p)$.
Then the operators defined in \eqref{eq:G-one-dim} and \eqref{eq:se.sdd.2} coincide.

\smallbreak
We proceed by verifying the assumptions of Corollary~\ref{cor:gen-expl-an}. Firstly, by
\cite[Thm.4.2]{CM:98} the operator $A$ generates an analytic semigroup of angle $\frac\pi2$ on $X$. Hence, Assumption~\ref{assump:Dir-op-ex}.(a) is satisfied.

\smallbreak
(i) As shown in the proof of \cite[Cor.~4.1, part~(ii)]{EF:05} the operator $A_0\subset\Am$ with domain
\[
D(A_0):=\Bigl\{f\in D(\Am):\tbinom{f(0)}{f(1)}=0\Bigr\}
\]
is dissipative and invertible.  Moreover, $[0,+\infty)\subset\rho(A_0)$ and $\|A_0R(\lambda,A_0)\|\le2$ for all $\lambda\ge0$.
Next, let $\eps_0(s):=1-s$ and $\eps_1(s):=s$ for $s\in[0,1]$. Then $\tilde L_0\in\sL(\dX,X)$ where $\tilde L_0\binom{x_0}{x_1}:=x_0\cdot\eps_0+x_1\cdot\eps_1\in\ker(\Am)$%
\footnote{$\tilde L_0$ is just the abstract Dirichlet operator for $\Am$ and the (Dirichlet-type) boundary operator $\tilde L:=\binom{\delta_0}{\delta_1}$.}%
. For $\lambda>0$ define
\[
L_\lambda:=-\tfrac1\lambda\cdot A_0R(\lambda,A_0)\tilde L_0\in\sL(\dX,X).
\]
Then essentially the same computations as in the proof of Proposition~\ref{prop:L_lambda-rhoA} show that
$L_\lambda$ is indeed the abstract Dirichlet operator for $\Am$ and the (Wentzell-type) boundary operator $L=\binom{\delta_0\Am}{\delta_1\Am}$. Therefore, Assumption~\ref{assump:Dir-op-ex}.(b) is satisfied. Moreover,
$
\|L_\lambda\|\le\frac2\lambda\cdot\|L_0\|
$
for all $\lambda>0$ and hence $\rg(L_\lambda)\subseteq\Fav_1^A$ by Lemma~\ref{lem:char-1-admiss}.(f). This shows (i) for $\beta=0$.

\smallbreak
(ii) As in the proof of \cite[Cor.~4.1, part~(iii)]{EF:05} Taylor's formula implies
for $f\in D(A)$, $s\in[0,1]$ and $0\ne\eps\in(-1,1)$ such that $s+\eps\in[0,1]$ the estimate
\begin{align*}
\bigl|f'(s)\bigr|
&\le\frac{2}{|\eps|}\cdot\l\|f\r\|_\infty
   +\biggl|\,\int_s^{s+\eps}\frac{dr}{a(r)}\,\biggr|\cdot\|A f\|_\infty\\
&\le\frac{2}{|\eps|}\cdot\l\|f\r\|_\infty
   +M\cdot|\eps|^\delta\cdot\|A f\|_\infty
\end{align*}
for some $M\ge0$, where in the second inequality we used the assumption on the Hölder continuity of the anti-derivative of $\frac1a$. Choosing $\rho:=|\eps|^{-(1+\delta)}>1$ and $\alpha:=\frac1{1+\delta}\in[\frac12,1)$ we obtain for $f\in D(A)$
\[
\bigl\|\dds f\bigr\|_\infty
\le(M+2)\cdot\bigl(\rho^\alpha\|f\|_\infty+\rho^{\alpha-1}\|Af\|_\infty\bigr)
\]
for  all $\rho>1$. Since $(\frac d{ds},\Cene)$ is closed, Lemma~\ref{lem:RR-11.39} implies (ii) for all $\gamma\in(\alpha,1)\ne\emptyset$.

\smallbreak
(iii) follows since $\beta=0$ and $\gamma<1$.

\smallbreak
Summing up, we verified all assumptions of Corollary~\ref{cor:gen-expl-an} and hence $G$ generates an analytic semigroup of angle $\frac\pi2$ on $X$.
Finally, since at the beginning of the proof we showed that
$D(\Am)\subset\Cene$, we conclude by the closed graph and the
Arzel\`{a}--Ascoli theorems that
\[
X_1^G\inc\Cene\incc\Cne,
\]
where 
``$\incc$" denotes a compact injection.
Hence, $X_1^G\incc X$ and \cite[Prop.II.4.25]{EN:00} implies
that $G$ has compact resolvent. By \cite[Thm.~II.4.29]{EN:00} it follows that the semigroup generated by $G$ is compact. This completes the proof.
\end{proof}

\begin{cor}
Under the assumptions on $a,b,c\in\Cne$ made in Theorem~\ref{thm:expl-Wentzell}, the degenerate differential equation with generalized Wentzell boundary conditions given by
\[\left\{\tag{DE}
\begin{aligned}
		&\frac{du}{dt} (t,s) = a(s)\,\frac{d^2 u}{ds^2} (t,s) + b(s)\,\frac{d u}{ds}(t,s)+c(s)\, u(t,s), && 0< s<1, \;t\geq 0, \\
		&\frac{a\,d^2u}{ds^2} (t,j) = \phi_ju(t), &&j=0,1,\, t\geq 0, \\
		&u(0,s) = f_0(s), && 0\le s\le1
\end{aligned}
\right.
\]
is well posed on $\Cne$ for all functionals $\phi_0,\phi_1\in{(\Cene)}'$.
\end{cor}

\begin{rem}\label{rem:W-BC}
Note that every function $a\in\Cne$ of the form
\[
a(s)=m(s)\cdot s^{\alpha_0}(1-s)^{\alpha_1},\  s\in[0,1]
\]
for $\alpha_{0,1}\in[0,1)$ and a strictly positive $m\in\Cne$ satisfies the assumption of Theorem~\ref{thm:expl-Wentzell}. Hence, this result generalizes \cite[Cor.~4.1]{EF:05} and \cite[Thm.~3]{FGGR:02} where such $a$ and less general boundary operators $\Phi$ were considered, respectively.
\end{rem}

\subsection{A Reaction-Diffusion Equation on $\mathbf{L^p[0,\pi]}$  with nonlocal Neumann Boundary Conditions}\label{sec:rde}

For $1\le p<+\infty$ define $X^p:=\rL^p[0,\pi]$. Then it is well-known (or use \cite[Thm.~2.2.(b)]{CKW:08} and \cite[Thm.~3.14.17]{ABHN:01}) that the operator $A\subset\frac{d^2}{ds^2}$ with domain
\[
D(A):=\bigl\{f\in\rW^{2,p}[0,\pi]:f'(0)=0=f(\pi)\bigr\}
\]
generates a bounded analytic semigroup on $X^p$. For $\gamma\in(0,1)$ we take the space
\[
Z_\gamma:=\bigl[D\bigl((-A)^{\gamma}\bigr)\bigr].
\]
Moreover, let $Y^p:= \rL^ p([-\pi,0],X^p)$ which by \cite[Thm.~A.6]{BP:05} is isometrically isomorphic to $L^p([-\pi,0]\times[0,\pi])$. For this reason in the sequel we will use  the notation $v(r,s):=(v(r))(s)$ for $v\in Y^p$ and $r\in[-\pi,0]$, $s\in[0,\pi]$.
Then the following holds.

\begin{thm}\label{thm:expl.2}
Let $p\in[1,+\infty)$ and $\gamma\in(0,\frac1p)$. Then for every $P\in\sL(Z_\gamma,X^p)$ and all functions
$\mu: [-\pi,0] \rightarrow \mathbb{R}$ of bounded variation the operator
\begin{align*}
	\sG &:= \begin{pmatrix} \frac{\partial^2}{\partial s^2} + P & 0 \\ 0 & \frac{\partial}{\partial r} \end{pmatrix}, \\
	D(\sG) &:= \Bigl\{ \tbinom fv \in \rW^ {2,p}[0,\pi] \times \rW^ {1,p}\bigl([-\pi,0],X^p\bigr): v(0) = f, \,  f(\pi) = 0,\\ &\hspace{7.2cm} f'(0) = \int_0^{\pi} \int_{-\pi}^0 v(r,s) \dmr \ds \Bigr\}
\end{align*}
generates a $C_0$-semigroup on $\sX^p:=X^p\times Y^p=\rL^p[0,\pi]\times\rL^ p([-\pi,0],X^p)$.
\end{thm}

\begin{rem}\label{rem:an/non-an}
Note that the semigroup group generated by $\sG$ will \emph{not} be analytic. Nonetheless, Theorem~\ref{lem:A^alpha-admiss-pair} for analytic semigroups will be very helpful to deal with the analytic part in the first component. The whole matrix will then be treated by Corollary~\ref{cor:SW-gen}. 
\end{rem}

\begin{proof}
We first show how the operator $\sG$ fits into the abstract framework from Section~\ref{subsec:GE}. To this end we introduce the following operators and spaces, where, for simplicity, in the sequel we put $X:=X^p$, $Y:=Y^p$ and $\sX:=\sX^p$. Consider
\begin{itemize}
\item $A_m:= \frac{d^2}{ds^2}$ with domain $D(A_m) =\{f\in \rW^ {2,p}[0,\pi]:f(\pi)=0\}$ on $X$,
\item $L:=\delta'_0:D(A_m)\to\partial X:=\CC$, i.e., $Lf=f'(0)$,
\item $D_m:= \frac{d}{dr}$ with domain $D(D_m) =\rW^ {1,p}([-\pi,0],X)$ on $Y=\rL^  p([-\pi,0],X)$,
\item $K:=\delta_0:D(D_m)\to\dY:=X=\rL^p[0,\pi]$, i.e., $Kv=v(0)$,
\item $A=A_m|_{\ker L}$, $D:=D_m|_{\ker K}$.
\end{itemize}
Then, as mentioned above, $A$ is the generator of an analytic semigroup $\Tt$ on $X$ while $D$ generates the nilpotent left-shift semigroup $\St$ on $Y$. Moreover, the associated Dirichlet operators exist for $\mu=0$ and are given by
\begin{itemize}
\item $L_0\in\sL(\dX,X)=\sL(\CC,\rL^{p}[0,\pi])$, $(L_0 x)(s)=x\cdot(s-\pi)$ for $s\in[0,\pi]$,
\item $K_0\in\sL(\dY,Y)=\sL(\rL^{p}[0,\pi],\rL^  p([-\pi,0],\rL^{p}[0,\pi]))$, $(K_0 f)(r):=f$ for $r\in[-\pi,0]$.
\end{itemize}
This shows that Assumption~\ref{assump:Dir-op-ex} is satisfied.
Next, we define the spaces $\dsX:=\dX\times\dY=\CC\times\rL^{p}[0,\pi] $ and
$\sZ:=Z_\gamma\times[D(D_m)]=[D((-A)^\gamma)] \times \rW^ {1,p}([-\pi,0],X)$ and introduce the operator matrices
\begin{align*}
	&\sA := \begin{pmatrix} A & 0 \\ 0 & D \end{pmatrix}:D(\sA) := D(A) \times D(D)\subset\sX\to\sX, &&
	\sL_\sA := \begin{pmatrix} \LA & 0 \\ 0 & \KD \end{pmatrix} :\sX\rightarrow \sXmesA, \\
	&\sP:= \begin{pmatrix} P & 0 \\ 0 & 0\end{pmatrix}: \sZ \rightarrow \sX&&
	\Phi:= \begin{pmatrix}0 & \phi \\ \Id & 0 \end{pmatrix}: \sZ \rightarrow \dsX,
\end{align*}
where $\KD:=-D_{-1}K_0$ and $\phi(v ):= \int_0^{\pi} \int_{-\pi}^0 v(r,s) \dmr \ds$.
Then as in Lemma~\ref{lem:G-as-SWp} we can write $(\sG,D(\sG))$ as a perturbation of the form
\begin{align*}
	\sG = (\sA_{-1} +\sP+\sL_\sA\cdot\Phi)|_\sX.
\end{align*}
We proceed by verifying the conditions of Corollary~\ref{cor:SW-gen}.
Since $\sA$ is diagonal with diagonal domain, we can split the problem into three parts: We show that
\begin{enumerate}[(i)]
\item $(\LA, P)$ is jointly $p$-admissible for $A$,
\item $(\KD, \phi)$ is jointly $p$-admissible for $D$,
\item $1\in\rho(\Ft)$ for some $t>0$ where
\begin{equation}\label{eq:F_t-simple}
  \Ft = 	\begin{pmatrix}
  				\Ft^{(A,\Id,P)} & \Ft^{(A,\LA,P)} & 0 \\
      			0 & 0 & \Ft^{(D,\KD,\phi)} \\
     			\Ft^{(A, \Id, \Id)} & \Ft^{(A,\LA, \Id)} & 0
    		\end{pmatrix}.
\end{equation}
\end{enumerate}

As already mentioned, due to the non-analytic part stemming from the left-shift-semigroup $\St$ generated by $D$ on $Y$,  the operator matrix $\sG$ will not generate an analytic semigroup on $\sX$. Still, we use of our perturbation Theorem~\ref{lem:A^alpha-admiss-pair} to treat the analytic part~(i) and also to prove (iii).

\smallbreak
(i) First we use Lemma~\ref{lem:char-1-admiss} to show that $\rg(L_0) \subseteq \mathrm{Fav}_{\frac{p+1}{2p}}^A$. To this end note that for $c\in\dX=\CC$ we have for $\lambda>0$
\[
\l(L_{\lambda} c\r)(s) = c\cdot\frac{\sinh\bigl(\sqrt{\lambda}(s-\pi)\bigr)}{\sqrt{\lambda}\cdot \cosh(\pi\sqrt{\lambda})},\quad s\in[0,\pi].
\]
Using this representation and the estimate $\sinh(s)\le\frac{e^s}2$ for $s\ge0$ we obtain
\begin{align*}
\sup_{\lambda > 0} \bigl\| \lambda^{\frac{p+1}{2p}} L_{\lambda}\bigr\|
&=\sup_{\lambda>0}\lambda^{\frac{1}{2p}}\cdot\Biggl(\int_0^\pi
\biggl(\frac{\sinh\bigl(\sqrt{\lambda}(s-\pi)\bigr)}{\cosh(\pi\sqrt{\lambda})}\biggr)^p\ds\Biggr)^{\frac1p}\\
&\leq \sup_{\lambda > 0} \frac{e^{\pi\sqrt{\lambda}}}{2p^{\frac{1}{p}}\cdot \cosh(\pi\sqrt{\lambda})} \le 1.
\end{align*}
This implies assumption~(i) of Theorem~\ref{lem:A^alpha-admiss-pair} for $\beta =1-\frac{p+1}{2p}= \frac{p-1}{2p}$.
Moreover, since $\gamma<\frac1p$,
\[
\beta+\gamma<\tfrac{p-1}{2p}+\tfrac1p=\tfrac{p+1}{2p}\le1,
\] 
hence Theorem~\ref{lem:A^alpha-admiss-pair} implies that $(\LA, P)$ is jointly $q$-admissible for $A$  for all $q\in(\frac{2p}{p+1},\frac1\gamma)$. If $p=1$, then $\beta=0$, hence again by Theorem~\ref{lem:A^alpha-admiss-pair} we obtain the same conclusion for $q=1$. For $p>1$ we have $p\in(\frac{2p}{p+1},\frac1\gamma)$. This proves (i).
 
\smallbreak
(ii) We first show that the functional $\phi$ is a $p$-admissible observation operator for $D$. In fact, for $v\in D(D)= \rW_0^{1,p}([-\pi,0],X)$ we have
\begin{align}
\int_0^{\pi} \bigl|\phi\, S(t)v\bigr|^p\dt
&=\int_0^{\pi} \left| \int_0^{\pi} \int_{-\pi}^{-t}  v(r+t,s)\dmr\ds \right|^p dt \nonumber \\
&\le \int_0^{\pi} \biggl( \int_0^{\pi} \int_{-\pi}^{-t} \bigl|v(t+r,s)\bigr| \dnmr\ds \biggr)^p\dt \nonumber \\
&\leq \int_0^{\pi} \bigl(\pi\cdot|\mu|[-\pi,-t]\bigr)^{p-1}  \int_0^{\pi} \int_{-\pi}^{-t}  \bigl| v(t+r,s)\bigr|^p \dnmr\ds \dt \label{Ho} \\
&\leq  \bigl(\pi\cdot|\mu|[-\pi,0]\bigr)^{p-1}\cdot\int_0^{\pi}  \int_{-\pi}^{-t} \int_{0}^{\pi} \bigl| v(t+r,s)\bigr|^p\ds  \dnmr\dt \label{Fubini1} \\
&= \bigl(\pi\cdot|\mu|[-\pi,0]\bigr)^{p-1}\cdot \int_0^{\pi} \int_{-\pi}^{-t}  \left\| v(t+r)\right\|_{X}^p  \dnmr\dt \nonumber \\
&=\bigl(\pi\cdot|\mu|[-\pi,0]\bigr)^{p-1}\cdot \int_{-\pi}^{0} \int_{0}^{-r}  \left\| v(t)\right\|_{X}^p\dt \dnmr \label{Fubini2} \\
&\leq \pi^{p-1}\cdot\bigl(|\mu|[-\pi,0]\bigr)^{p}\cdot\left\|v\right\|_{p}^p, \nonumber
\end{align}
where in \eqref{Ho} we used Hölder's inequality twice and the Fubini--Tonelli theorem in \eqref{Fubini1}, \eqref{Fubini2}.

\smallbreak
Next, as in the proof of \cite[Cor.~25]{ABE:13} we obtain for $u \in W_{0}^{1,p}([0,\pi],X)$ and $t\in(0,\pi]$ that
\begin{equation}\label{eq:CoMap-Ex-HMR}
\int_0^{t} S_{-1}(t-r)\KD u(r)\dr = u\bigl(\max\{0, \p+t\}\bigr).
\end{equation}
Hence, the operator $\KD$ is a $p$-admissible control operator for $D$.

\smallbreak
Now we show that the pair $(\KD, \phi)$ is $p$-admissible. In fact, using \eqref{eq:CoMap-Ex-HMR} we obtain for $u\in\rW^{1,p}([0,\pi],X)$ by essentially the same computations as above 
\begin{align*}
\int_0^{\pi} \biggl|\phi\int_0^{t} S_{-1}(t-r)\KD u(r)\dr\biggr|^p dt
&=\int_0^{\pi} \left| \int_0^{\pi} \int_{-\pi}^{0}  u\bigl(\max\{0, r+t\},s\bigr)\dmr\ds \right|^p dt \nonumber \\
&\le \int_0^{\pi} \biggl( \int_0^{\pi} \int_{-t}^0 \bigl|u(t+r,s)\bigr| \dnmr\ds \biggr)^p\dt \nonumber \\
&\le \bigl(\pi\cdot|\mu|[-\pi,0]\bigr)^{p-1}\cdot \int_0^{\pi} \int_{-t}^{0}  \left\| u(t+r)\right\|_{X}^p  \dnmr\dt \nonumber \\
&\leq \pi^{p-1}\cdot\bigl(|\mu|[-\pi,0]\bigr)^{p}\cdot\left\| u\right\|_{p}^p. \nonumber
\end{align*}

\smallbreak
(iii) We first determine the input-output map of the complete system. Note that by \eqref{eq:def-U} we should choose $\sU=\sX\times\dsX=(X\times Y)\times(\dX\times\dY)$. However, since the perturbation $P$ only acts on $X$ we can cancel out the factor $Y$ and choose
\[
\sU:=X\times\dsX=X\times(\dX\times\dY)=\rL^ p[0,\pi]\times\CC\times\rL^ p[0,\pi].
\]
By this reduction and the diagonal structure of the generator $\sA$ we obtain the simplified input-output map%
\footnote{The real input-output map of the whole system is obtained from $\Ft$ by inserting at the second place a row and a column containing only zeros.}
$\Ft\in\sL(\sU)$ given by \eqref{eq:F_t-simple}.
Note that by Theorem~\ref{lem:A^alpha-admiss-pair}.(e) we have
\begin{equation}\label{eq:est-sFt-A}
\bigl\|\Ft^{(A,*,\star)}\bigr\|\to0\text{ as }t\to0^+,
\end{equation}
where ``$*,\star$'' indicates one of the pairs ``$\Id,P\,$'', ``$\LA,P\,$'', ``$\Id,\Id\,$'' or ``$\LA,\Id\,$''.
In particular, $\Id - \Ft^{(A,\Id,P)}$ is invertible for $t>0$ sufficiently small. Using Schur complements (cf. \cite[Lem.~2.1]{Nag:89B}) the invertibility of $\Id - \Ft$ is therefore equivalent to the invertibility of
\begin{align*}
	&\Id -
		\begin{pmatrix} \Ft^{(A, \Id, \Id)} & \Ft^{(A,\LA, \Id)} \end{pmatrix}
		\cdot\begin{pmatrix} \bigl(\Id - \Ft^{(A,\Id,P)}\bigr)^{-1} & \Ft^{(A,\LA,P)} \\
      			0 & \Id \end{pmatrix}
      	\cdot\begin{pmatrix} 0 \\ \Ft^{(D,\KD,\phi)} \end{pmatrix} \\
      &\qquad= \Id - \Ft^{(A, \Id, \Id)}\cdot \Ft^{(A,\LA, P)}\cdot \Ft^{(D,\KD,\phi)} - \Ft^{(A,\LA, \Id)}\cdot \Ft^{(D,\KD,\phi)}.
\end{align*}
Since by \eqref{eq:est-sFt-A}
\begin{align*}
	\left\| \Ft^{(A, \Id, \Id)}\cdot \Ft^{(A,\LA, P)}\cdot \Ft^{(D,\KD,\phi)} + \Ft^{(A,\LA, \Id)}\cdot \Ft^{(D,\KD,\phi)} \right\|
		\to0\text{ as }t\to0^+,
\end{align*}
the assertion holds for $t>0$ sufficiently small. 

\smallbreak
Summing up (i)--(iii), by Corollary~\ref{cor:SW-gen} the matrix $\sG$ generates a $C_0$-semigroup on $\sX$.
\end{proof}

\begin{cor}
The reaction-diffusion equation subject to  Neumann boundary conditions with distributed unbounded delay given by
\[\left\{\tag{RDE}
\begin{aligned}
		&\frac{du}{dt} (t,s) = \frac{d^2 u}{ds^2} (t,s) + b(s)\, \frac{d u}{ds}(t,s)+c(s)\,u(t,s), && 0 < s < \pi, \;t\geq 0, \\
		&\frac{d u}{d s} (t,0) =  \int_0^{\pi}\int_{-\pi}^0 u(t+r,s) \dmr\ds, && t\geq 0, \\
		&u(t,\pi) = 0, && t\geq 0,\\
		&u(r,s) = u_0(r,s), && 0 < s < \pi, \;r\in [-\pi,0],\\
		&u(0,s) = f_0(s), && 0 < s < \pi
\end{aligned}
\right.
\]
is well posed on $\rL^p[0,\pi]$ for all $p\in[1,2)$, $b,c\in\rL^\infty[0,\pi]$ and $\mu$ of bounded variation.
\end{cor}

\begin{proof} Let $P:=b(\p)\dds+c(\p)$ with domain $Z:=\rW^ {1,p}[0,\pi]$.
By the previous result it suffices to prove that $Z_\gamma\inc Z$ for  $\gamma\in(\frac{1}{2},\frac1p)\ne\emptyset$. To this end note that by  \cite[Ex.~III.2.2]{EN:00} for each $f\in D(A)$ and $\eps>0$ 
\begin{align*}
	\bigl\| f' \bigr\|_p \leq \tfrac{9}{\eps}\cdot \| f \|_p + \eps\cdot \| A f \|_p.
\end{align*}
Setting $\rho:=\eps^{-2}$ the assertion follows from Lemma~\ref{lem:RR-11.39}.
\end{proof} 

\begin{rem}\label{rem:HMR-ex}
Theorem~\ref{thm:expl.2}
generalizes \cite[Expl.~5.2]{HMR:15} where only the exponent $p=2$ and the perturbation $P=0$ is considered.
\end{rem}

\section{Conclusions}

We presented an approach to the perturbation of generators $A$ of analytic semigroups which, under  conditions on $\rg(B)$ and $D(C)$, 
gives that $(\Ame+BC)|_X$ generates an analytic semigroup as well. Our main results are Theorem~\ref{lem:A^alpha-admiss-pair} for the general case and Corollaries~\ref{cor:SW-gen}, \ref{cor:gen-expl-an} for our Generic Example~\ref{subsec:GE}. In contrast to the known literature on this subject, our results
\begin{itemize}
\item allow perturbations $P=BC$ which are not relatively $A$-bounded as, e.g., in \cite[Sect.~III.2]{EN:00},
\item always give the angle of analyticity of the perturbed semigroup, unlike \cite[Thm.~2.6]{GK:91} in the situation of the generic example,
\item are applicable also to coupled systems which are only in part governed by an analytic semigroup, cf. Remark~\ref{rem:an/non-an}.
\end{itemize}
Moreover, nevertheless being very general, our results applied to concrete situations recover or even improve generation results obtained by methods tailored for the specific case, cf. Remarks~\ref{rem:G-K}, \ref{rem:W-BC} and \ref{rem:HMR-ex}.

\appendix
\section{}\label{app:admiss}

In this appendix we collect some results which are quite helpful to check the hypotheses  of Theorem~\ref{lem:A^alpha-admiss-pair} and Corollaries~\ref{cor:SW-gen}, \ref{cor:gen-expl-an} in applications. First we consider condition~(i) in Corollary~\ref{cor:gen-expl-an}.

\begin{lem}\label{lem:char-1-admiss} In the situation of Subsection~\ref{subsec:GE}, 
for $\alpha\in(0,1]$ the following are equivalent.
\begin{enumerate}
\item There exists $\lambda_0>\gb(A)$ such that $\sup_{\lambda>\lambda_0}\|\lambda^{\alpha}L_\lambda x\|<+\infty$ for all $x\in\dX$.
\item There exist $\lambda_0>\gb(A)$ and $M>0$ such that $\|Lx\|\ge\lambda^{\alpha}M\cdot\|x\|$ for all $\lambda\ge\lambda_0$ and $x\in\ker(\lambda-\Am)$.
\item $\rg(L_\mu)=\ker(\mu-A_m)\subset\Fav_\alpha^A$ for some $\mu\in\rho(A)$.
\end{enumerate}
Moreover, if $\alpha=1$, then (a)--(c) are also equivalent to
\begin{enumerate}
\item[(d)] $\LA$ is a $1$-admissible control operator for $A$.
\item[(e)] For all $x\in D(A_m)$ there exists a sequence $(x_n)_{n\in\NN}\subset D(A)$ such that $\lim_{n\to+\infty}x_n=x$ and $\sup_{n\in\NN}\|x_n\|<+\infty$.
\end{enumerate}
\end{lem}

\begin{proof}  The equivalence of (a) and (b) follows immediately from the definition of $L_\lambda$ as the inverse of $L:\ker(\lambda-A)\to\dX$. To show the equivalence of (a) and (c), note that from  \eqref{eq:def-L_mu} we obtain for $x\in\dX$ and fixed $\mu\in\rho(A)$
\begin{align*}
\sup_{\lambda>\lambda_0}\bigl\|\lambda^\alpha A\RlA L_\mu x\bigr\|
&\le\sup_{\lambda>\lambda_0}\bigl\|\lambda^\alpha (\mu-A)\RlA L_\mu x\bigr\|+
\sup_{\lambda>\lambda_0}\bigl\| \mu\lambda^\alpha\RlA L_\mu x\bigr\|\\
&=\sup_{\lambda>\lambda_0}\bigl\|\lambda^\alpha L_\lambda x\bigr\|+\sup_{\lambda>\lambda_0}\bigl\| \mu\lambda^\alpha\RlA L_\mu x\bigr\|.
\end{align*}
Recall that $\alpha\le1$, hence by the Hille--Yosida theorem we have in any case
\[
\sup_{\lambda>\lambda_0}\bigl\| \mu\lambda^\alpha\RlA L_\mu x\bigr\|<+\infty.
\]
Thus, we conclude that
\[
\sup_{\lambda>\lambda_0}\bigl\|\lambda^\alpha A\RlA L_\mu x\bigr\|<+\infty
\quad\iff\quad
\sup_{\lambda>\lambda_0}\bigl\|\lambda^\alpha L_\lambda x\bigr\|<+\infty.
\]
By \cite[Prop.~II.5.12]{EN:00} the condition on the left-hand-side is equivalent to $L_\mu x\in\Fav_\alpha^A$ and therefore we obtain (a)$\iff$(c).

\smallbreak
For the equivalence of (a)--(c) and (d)--(e) in case $\alpha=1$ see \cite[Proof of Cor.~III.3.6]{EN:00},  \cite[Ex.~III.3.8.(4)]{EN:00} and \cite[Ex.~II.5.23.(2)]{EN:00}. 
\end{proof}

If one can represent a Banach space $Z$ as the domain $[D(K)]$ of a closed operator $K$ equipped with its graph norm, then the condition $[D((\lambda-A)^\gamma)]\inc Z$ appearing in Theorem~\ref{lem:A^alpha-admiss-pair}.(ii) and Corollary~\ref{cor:gen-expl-an}.(ii) can frequently be verified by the following result.

\begin{lem}\label{lem:RR-11.39}
Let $A$ be the generator of an analytic semigroup and let $K$ be a closed linear operator such that $D(A)\subseteq Z=[D(K)]$. If for $\alpha\in(0,1)$ and every $\rho\ge\rho_0>0$ we have
\begin{equation*} 
\|Kx\|\le M\cdot\bigl(\rho^\alpha\|x\|+\rho^{\alpha-1}\|Ax\|\bigr)
\quad\text{for all }x\in D(A)
\end{equation*}
and some constant $M\ge0$, then $[D((\lambda-A)^\gamma)]\inc Z$ for every $\gamma>\alpha$ and $\lambda>\gb(A)$.
\end{lem}

\begin{proof}
By (the proof of) \cite[Lem.~11.39]{RR:93} the operator $K(\lambda-A)^{-\gamma}$ is bounded which by the closed graph theorem implies that $D((\lambda-A)^\gamma)\subseteq D(K)$. Moreover, from the estimate
\[
\|Kx\|\le\bigl\|K(\lambda-A)^{-\gamma}\bigr\|\cdot\|(\lambda-A)^\gamma x\|
\]
it follows that  $[D((\lambda-A)^\gamma)]\inc Z$.
\end{proof}

Finally, we consider pairs $(B,C)$ where we assume that one of the operators $B$ or $C$ is bounded.

\begin{lem}\label{lem:norm-sFt}
Let  $A$ be the generator of a $C_0$-semigroup and let $p\ge1$.
\begin{enumerate}
\item If $B\in\sL(U,\XmeA)$ is a $p$-admissible control operator and $Z=X$, i.e., $C\in\sL(X,\Y)$ then $(A,B,C)$ is compatible, $(B,C)$ is jointly $p$-admissible and there exists $M\ge0$ and $\tn>0$ such that
\[
\bigl\|\Ft^{(A,B,C)}\bigr\|\le M\cdot t^{\frac1p}\quad\text{for all }0<t\le t_0.
\]
In particular, $Id_U$ is $p$-admissible for the triple $(A,B,C)$.
\item If $B\in\sL(U,X)$ and $C\in\sL(Z,\Y)$ is a $p$-admissible observation operator, then $(A,B,C)$ is compatible, $(B,C)$ is jointly $p$-admissible and there exists $M\ge0$ and $\tn>0$ such that
\[
\bigl\|\Ft^{(A,B,C)}\bigr\|\le M\cdot t^{1-\frac1p}\quad\text{for all }0<t\le\tn.
\]
In particular, if $p>1$ then $Id_U$ is $p$-admissible for the triple $(A,B,C)$.
\end{enumerate}
\end{lem}

\begin{proof}
The assertions follow immediately from \cite[Rems.~17 and 19]{ABE:13}.
\end{proof}

\end{document}